\documentclass[preprint,12pt]{article}

\usepackage{authblk}

\usepackage{mathptmx}
\usepackage[scaled=.95]{helvet}
\usepackage{courier}

\usepackage{graphicx}
\usepackage{amsmath}
\usepackage{amssymb}
\usepackage{amsthm} 
\usepackage{subfig}
\usepackage{xcolor}
\usepackage{mathrsfs}
\usepackage{esint}  

\usepackage{subfig}
\usepackage{rotating}

\usepackage{lineno}

\usepackage{hyperref}   
\hypersetup{pdftitle=Trend to equilibrium for a reaction-diffusion system modelling reversible enzyme reaction}
\hypersetup{pdfauthor=Jan Elias}

\newtheorem{theorem}{Theorem} [section] 
\newtheorem{lemma}[theorem]{Lemma} 

\theoremstyle{definition}

\theoremstyle{remark}

\def \isecp{i \in \{S,E,C,P\}}

\def \ns{n_S}
\def \nc{n_C}
\def \ne{n_E}
\def \np{n_P}

\def \nsR{n_{S,\infty}}
\def \ncR{n_{C,\infty}}
\def \neR{n_{E,\infty}}
\def \npR{n_{P,\infty}}

\def \nsRO{\overline{n_S}}
\def \ncRO{\overline{n_C}}
\def \neRO{\overline{n_E}}
\def \npRO{\overline{n_P}}

\def \NSO{\overline{N_S}}
\def \NCO{\overline{N_C}}
\def \NEO{\overline{N_E}}
\def \NPO{\overline{N_P}}
\def \NiO{\overline{N_i}}

\def \delSO{\overline{\delta_S^2}}
\def \delCO{\overline{\delta_C^2}}
\def \delEO{\overline{\delta_E^2}}
\def \delPO{\overline{\delta_P^2}}
\def \deliO{\overline{\delta_i^2}}

\def \NSI{N_{S,\infty}}
\def \NCI{N_{C,\infty}}
\def \NEI{N_{E,\infty}}
\def \NPI{N_{P,\infty}}
\def \NiI{N_{i,\infty}}

\newcommand{\re}{\mathbb{R}}
\newcommand{\dx}{\, \textnormal{d}x}

\newcommand{\dt}{\, \textnormal{d}t}

\title{Trend to equilibrium for a reaction-diffusion system modelling reversible enzyme reaction}
\author{J\'{a}n Elia\v{s}}
\affil{\small Laboratoire de Math\'{e}matiques d'Orsay, Univ. Paris-Sud, CNRS, Universit\'{e} Paris-Saclay, 91405 Orsay, France}
\affil{\small Sorbonne Universit\'{e}s, Inria, UPMC Univ Paris 06, Lab. J.L. Lions  UMR CNRS 7598, Paris, France}
\date{}                     

\begin{document}




\maketitle

\begin{abstract}
A spatio-temporal evolution of chemicals appearing in a reversible enzyme reaction and modelled by a four component reaction-diffusion system with the reaction terms obtained by the law of mass action is considered. The large time behaviour of the system is studied by means of entropy methods. \\

\end{abstract}

\section{Introduction}

In eukaryotic cells, responses to a variety of stimuli consist of chains of successive protein interactions where enzymes play significant roles, mostly by accelerating reactions. Enzymes are catalysts that facilitate a conversion of molecules (generally proteins) called substrates into other molecules called products, but they themselves are not changed by the reaction.
In the reaction scheme proposed by Michaelis and Menten~\cite{Michaelis-1913} in 1913, an enzyme $E$ converts a substrate $S$ into a product $P$ through a two step process, schematically written as
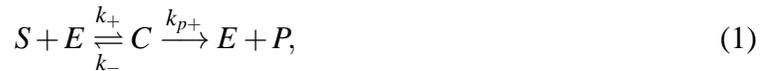
\begin{equation} \label{eq:ER:s1:00}
S + E \underset{k_{-}}{\overset{k_{+}}{\rightleftharpoons}} C  \overset{k_{p+}}{\longrightarrow} E + P,
\end{equation}
where $C$ is an intermediate complex and $k_{+}, k_{-}$ and $k_{p+}$ are positive kinetic rates of the reaction. In 1925, the enzyme reaction~(\ref{eq:ER:s1:00}) was analysed by Briggs and Haldane~\cite{Briggs-1925} by using ordinary differential equations (ODE) derived from mass action kinetics. In their quasi-steady state approximation (QSSA), the complex is assumed to reach a steady state quickly, i.e., there is no change in its concentration $\nc=[C]$ in time ($\textnormal{d} \nc/\textnormal{d} t = 0$). The analysis yields an algebraic expression, the so-called Michaelis-Menten function, for $\nc$ and a simple, though nonlinear, ODE for the substrate's concentration $\ns=[S]$. The kinetics of enzyme reactions described by Briggs and Haldane is sometimes called the Michealis-Menten kinetics. Further details on the existing approximation techniques can be found in \cite{Cornish-Bowden-2012, Schnell-2000, Schnell-2002, Segel-1989}, a validation of the QSSA also in \cite{Perthame-LN}.

Many important reactions in biochemistry are, however, reversible in the sense that a significant amount of the product $P$ exists in the reaction mixture due to a reaction of $P$ with the enzyme $E$,~\cite{Cornish-Bowden-2012}. Therefore, the Michaelis-Menten mechanism~(\ref{eq:ER:s1:00}) is incomplete for these reactions and should be rather replaced by
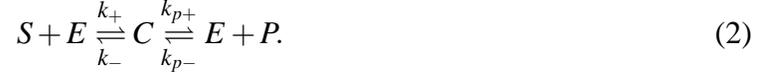
\begin{equation} \label{eq:ER:s2:09}
S + E \underset{k_{-}}{\overset{k_{+}}{\rightleftharpoons}} C  \underset{k_{p-}}{\overset{k_{p+}}{\rightleftharpoons}} E + P.
\end{equation}


Almost entire mathematical modelling of the enzyme reactions~(\ref{eq:ER:s1:00}) and~(\ref{eq:ER:s2:09}) is usually done by using ODE approaches, \cite{Cornish-Bowden-2012, Schnell-2000, Schnell-2002, Segel-1989}. However, protein pathways occur in living cells, (heterogenous) spatial structure of which has an impact on the enzyme efficiency and the speed of enzyme reactions. In this paper, a spatial reaction-diffusion system for the reversible enzyme reaction~(\ref{eq:ER:s2:09}) is studied without any kind of approximation.
More precisely, we will consider a system of four equations for the concentrations $n_i$, $\isecp$, for the species appearing in~(\ref{eq:ER:s2:09}) with the reaction terms obtained by the law of mass action. Moreover, we assume that the species can diffuse freely and randomly (modelled by linear diffusion) with constant diffusion rates. Thus, we consider the system 
\begin{equation} \label{eq:ER:s2:10} \begin{aligned}
\dfrac{\partial\ns}{\partial t} - D_S \Delta \ns & = k_{-} \nc - k_{+} \ns \ne,\\
\dfrac{\partial \ne}{\partial t} - D_E \Delta \ne  &= (k_{-}+k_{p+}) \nc - k_{+} \ns \ne - k_{p-}\ne\np,\\ 
\dfrac{\partial\nc}{\partial t} - D_C \Delta \nc  & = k_{+} \ns \ne - (k_{-}+k_{p+}) \nc + k_{p-}\ne\np, \\
\dfrac{\partial\np}{\partial t}  - D_P \Delta \np  &= k_{p+} \nc - k_{p-}\ne\np.
\end{aligned}\end{equation}
It is assumed that for each $\isecp$ $n_i = n_i(t,x)$ is defined on an open, bounded domain $\Omega \subset \re^d$ with a sufficiently  smooth boundary $\partial \Omega$ (e.g., $C^2$) and a time interval $I = [0,T]$ for $0<T<\infty$, $Q_T = I \times \Omega$. Without loss of generality, we assume $\vert \Omega \vert = 1$. The diffusion coefficients $D_i$ are supposed to be positive constants, possibly different from each other. Further, we assume that there exist nonnegative measurable functions $n_i^0$ such that 
\begin{equation} \label{eq:ER:s2:11}
n_i(0,x) = n_i^0(x) \; \text{ in } \Omega, \quad \int_{\Omega} n_i^0(x) \dx > 0, \quad \forall \, \isecp.
\end{equation}
Finally, the system is coupled with the zero-flux boundary conditions 
\begin{equation} \label{eq:ER:s2:12}
\nabla n_i \cdot \nu = 0, \quad \forall t \in I, \; x \in \partial \Omega, \;  \isecp,
\end{equation}
where $\nu$ is a unit normal vector pointed outward from the boundary $\partial \Omega$. 

Two linearly independent conservation laws can be observed, in particular, 
\begin{eqnarray}
& {\displaystyle \int_{\Omega}} (\ne + \nc)(t,x) \dx = {\displaystyle \int_{\Omega}} (\ne^0+\nc^0)(x) \dx = M_1, \label{eq:ER:s2:02} \\
& {\displaystyle \int_{\Omega}} (\ns + \nc + \np)(t,x) \dx = {\displaystyle \int_{\Omega}} (\ns^0 + \nc^0 + \np^0) (x) \dx = M_2, \label{eq:ER:s2:03}
\end{eqnarray}
for each $t \ge 0$, where $M_1 >0, M_2 > 0$.  
Note that there is often $M_1 \ll M_2$  \cite{Cornish-Bowden-2012}, however, we will not assume any relation between $M_1$ and $M_2$.

The conservations laws~(\ref{eq:ER:s2:02}) and~(\ref{eq:ER:s2:03}) imply the uniform $L^1$ bounds on the solutions of~({\ref{eq:ER:s2:10}})-({\ref{eq:ER:s2:12}}) which are insufficient for the existence of global solutions. 
A global weak solution in all space dimensions ($d \ge 1$), however, can be deduced from a combination of a duality argument (see Appendix), which provides estimates on the (at most quadratic) nonlinearities of the system, and an approximation method developed in \cite{PSch, DFPV}, which justifies rigorously the existence of the weak solution to~(\ref{eq:ER:s2:10})-({\ref{eq:ER:s2:12}}) builded up from the solutions of the approximated system. The existence of the global weak solution with the total mass conserved by means of~(\ref{eq:ER:s2:02}) and~(\ref{eq:ER:s2:03}) can be shown constructively by the semi-implicit (Rothe) method \cite{Roubicek-2013, EliasThesis}, a method suitable for numerical simulations. We also refer to~\cite{Bothe-2015} where a proof of the existence of the unique, global-in-time solution to~(\ref{eq:ER:s2:10})-({\ref{eq:ER:s2:12}}) with the concentration dependent diffusivities and $d \le 9$ is obtained by a combination of the duality and bootstrapping arguments.
Therefore, we do not give any rigorous results on the existence of solutions; 
instead, we focus on the large time behaviour of the solution to its equilibrium as $t \to \infty$. However, we derive a-priori estimates which make all the integrals that appear (e.g., entropy functional) well defined.

In particular, by a direct application of a duality argument (see Appendix), we deduce that whenever $n_i^0 \in L^2(\log L)^2(\Omega)$, then $n_i \in L^2(\log L)^2(Q_T)$ for each $0< T <\infty$, $\isecp$. 
With the $L^2(\log L)^2$ estimates at hand, the solution $n_i$ for $\isecp$ can be shown, as in \cite{Bothe-2015}, to belong to $L^{\infty}((0,\infty)\times \overline{\Omega})$ by using the properties of the heat kernel combined with a bootstrapping argument and by assuming sufficiently regular  initial data and $\partial \Omega$. Thus, we can deduce the global-in-time existence of the classical solution (that is bounded solution which has classical derivatives at least \emph{a.e.} and the equations in~({\ref{eq:ER:s2:10}}) are understood pointwise) by the standard results for reaction-diffusion systems \cite{Lady1968}. 

The main result of this paper is a quantitative analysis of the large time behaviour of the solution $n_i$, $\isecp$, to~({\ref{eq:ER:s2:10}})-({\ref{eq:ER:s2:12}}). It can be stated as follows:
\begin{theorem} \label{theorem_ER_02}
Let $(n_S, n_E, n_C, n_P)$ be a solution to~({\ref{eq:ER:s2:10}})-({\ref{eq:ER:s2:12}}) satisfying~(\ref{eq:ER:s2:02}) and~(\ref{eq:ER:s2:03}). Then there exist two explicitly computable constants $C_1$ and $C_2$ such that 
\begin{equation} \label{eq:ER:s2:27}
\sum_{i\in\{S,E,C,P\}} \Vert n_i - n_{i, \infty} \Vert_{L_1(\Omega)}^2 \le C_2 e^{ -C_1 t}
\end{equation}
where $n_{i, \infty}$ is the unique, positive, detailed balance steady state defined in~(\ref{eq:ER:s2:18}).
\end{theorem}
\noindent In other words we show $L^1$ convergence of the solution $n_i$, $\isecp$, of~({\ref{eq:ER:s2:10}})-({\ref{eq:ER:s2:12}}) to its respective steady state $n_{i,\infty}$, $\isecp$, at the rate $C_1/2$.

We remark that by following the general theory of the detailed balance systems, e.g., \cite{Fellner-2016a} and references therein, there exists a unique detailed balance equilibrium to the system~({\ref{eq:ER:s2:10}})-({\ref{eq:ER:s2:12}}) satisfying the conservation laws
\begin{equation} \label{eq:ER:s2:14}
\neR + \ncR = M_1, \quad \nsR + \ncR + \npR = M_2,
\end{equation}
and the detailed balance conditions
\begin{equation} \label{eq:ER:s2:15}
k_{-}\ncR = k_{+}\nsR\neR, \quad k_{p+}\ncR = k_{p-}\npR\neR.
\end{equation}
It is easy to show that the unique, strictly positive equilibrium $\mathbf{n}_{\infty} = (\nsR,\neR,\ncR,\npR)$ is then
\begin{equation}\label{eq:ER:s2:18} \begin{aligned}
& \ncR = \dfrac{1}{2}\left(M+K - \sqrt{(M+K)^2-4 M_1M_2} \right), \\ 
\neR  & = M_1 - \ncR, \quad \nsR = \dfrac{k_{-}\ncR}{k_{+}\neR}, \quad \npR = \dfrac{k_{p+}\ncR}{k_{p-}\neR},
\end{aligned} \end{equation}
where $M=M_1+M_2$ and $K = k_{-}/k_{+}+k_{p+}/k_{p-}$.

Theorem~\ref{theorem_ER_02} is proved by means of entropy methods, which are based on an idea to measure the distance between the solution and the stationary state by the (monotone in time) entropy of the system. This entropy method has been developed mainly in the framework of the scalar diffusion equations and the kinetic theory of the spatially homogeneous Boltzmann equation, see \cite{Arnold1, Carrillo-2001, Villani-2003} and references therein. The method has been already used to obtain explicit rates for the exponential decay to equilibrium in the case of reaction-diffusion systems modelling chemical reactions $2 A_1 \rightleftharpoons A_2$, $A_1 + A_2 \rightleftharpoons A_3$, $A_1 + A_2 \rightleftharpoons A_3 + A_4$ and $A_1 + A_2 \rightleftharpoons A_3 \rightleftharpoons A_4 + A_5$ in \cite{DF2, DF3, DF1, Fellner-2016a}. 
The large time behaviour of a solution to a general detailed balance reaction-diffusion system counting $R$ reversible reactions involving $N$ chemicals, 
\begin{equation}\label{eq:ER:r1:01} \alpha^j_1 A_1 + \ldots + \alpha^j_N A_N \rightleftharpoons \beta^j_1 A_1 + \ldots + \beta^j_N A_N \end{equation}
with the nonnegative stoichiometric coefficients $\alpha^j_1, \ldots, \alpha^j_N$, $\beta^j_1, \ldots, \beta^j_N$, for $j = 1, \ldots, R$, was also studied in \cite{Fellner-2016a}. However, the convergence rates could not be explicitly calculated without knowing explicit structure of the mass conservation laws in the general case. 

The present paper extends the application of the proposed entropy method for the reversible enzyme reaction~(\ref{eq:ER:s2:09}) counting two single reversible reactions. 
The difficulty comes from a chemical (enzyme) that appears in both reactions which makes~(\ref{eq:ER:s2:09}) different from the reaction $A_1 + A_2 \rightleftharpoons A_3 \rightleftharpoons A_4 + A_5$ studied in~\cite{Fellner-2016a}, in particular, in the structure of the conservation laws that is essential in the computation of the rates of convergence. Further, even though the convergence rates are obtained through a chain of rather simple but nasty calculations in \cite{DF2, DF3, DF1, Fellner-2016a}, we simplify them by means of an inequality~(\ref{eq:ER2:01}) in Lemma~\ref{lemma_ER_01}. In particular, if we denote $N_i = \sqrt{n_i}$, $N_{i,\infty} = \sqrt{n_{i,\infty}}$ and $\overline{N_i} = \int_{\Omega} N_i(x)\dx$ for some chemical $n_i$ and its equilibrium $n_{i,\infty}$, the expansion used in \cite{DF2, DF3, DF1, Fellner-2016a} (c.f., equation~(2.29) in \cite{Fellner-2016a}) to measure the distance between $\overline{N_i}$ and $N_{i,\infty}$ is of the form
\[\overline{N_i} = N_{i,\infty}(1+\mu_i) - \dfrac{\overline{N^2_i} - \overline{N_i}^2}{\sqrt{\overline{N^2_i}}+\overline{N_i}} \]
for some constant $\mu_i \ge -1$. The fraction in this expansion may become unbounded when $\overline{N^2_i}$ approaches zero, which has to be carefully treated. On the other hand, Lemma~\ref{lemma_ER_01} allows different expansions that consequently lead to easier calculations.

For the sake of completeness, a different approach based on a convexification argument is used in \cite{Mielke-2014} to study the large time behaviour of the reaction-diffusion system for~(\ref{eq:ER:r1:01}). However, it is difficult to derive explicit convergence rates even for a bit more complex chemical reactions such as~(\ref{eq:ER:s2:09}) by using this convexification argument.
First order chemical reaction networks have been recently analysed in \cite{Fellner-2015a}.

The rest of the paper is organised as follows. In Section \ref{sec:entropy} we introduce entropy and entropy dissipation functionals and provide first estimates including $L^2$ and $L^2(\log L)^2$ bounds. A main ingredient for the a-priori estimates is a duality argument that is presented in Appendix. The  large time behaviour of the solution as $t \to \infty$ studied by the entropy method is given in Section~\ref{sec:trend}. 

\section{Entropy, entropy dissipation and a-priori estimates}
\label{sec:entropy}

Let us first mention a simple result on the non-negativity of solutions of~({\ref{eq:ER:s2:10}})-({\ref{eq:ER:s2:12}}) which follows from the so-called quasi-positivity property of the right hand sides of~({\ref{eq:ER:s2:10}}), see~\cite{Pierre1}. 

\begin{lemma} \label{theorem_ER_03}
Let $n_i^0 \ge 0$ in $\Omega$, then $n_i \ge 0$ everywhere in $Q_T$ for each $\isecp$.
\end{lemma}

In the sequel, we will write shortly $\mathbf{n} = (\ns,\ne,\nc,\np)$. The entropy functional $E(\textbf{n}) : [0, \infty)^4 \to [0, \infty)$ and the entropy dissipation $D(\textbf{n}) : [0, \infty)^4 \to [0, \infty)$ are defined, respectively, by
\begin{equation} \label{eq:ER:s2:100}
E(\textbf{n}) = \sum_{i = \{S,E,C,P\}} \int_{\Omega} n_i\log (\sigma_i n_i) - n_i +1/\sigma_i \dx
\end{equation}
and
 \begin{equation} \label{eq:ER:s2:101} \begin{aligned}
D(\textbf{n}) = & \sum_{i = \{S,E,C,P\}} 4 D_i \int_{\Omega} \vert \nabla \sqrt{n_i} \vert^2 \dx\\
& + \int_{\Omega} \left[ \left( k_{+}\ns\ne - k_{-}\nc \right)\left( \log(\sigma_S \sigma_E \ns \ne) - \log(\sigma_C \nc) \right) \right. \\ & \left.+ \left( k_{p-}\ne\np - k_{p+}\nc \right)\left( \log(\sigma_E \sigma_P \ne \np) - \log(\sigma_C \nc) \right) \right] \dx,
\end{aligned}\end{equation}
where $\sigma_S$, $\sigma_E$, $\sigma_C$ and $\sigma_P$ depend on the kinetic rates.
The first integral of the entropy dissipation~(\ref{eq:ER:s2:101}) is known as the relative Fisher information in information theory and as the Dirichlet form  in the theory of large particle systems, since 
\[  4 \int \vert \nabla \sqrt{n_i} \vert^2 = \int \dfrac{\vert \nabla n_i \vert^2}{n_i} = \int n_i \left| \nabla \left( \log n_i \right) \right|^2,\] 
see \cite{Villani-2003}, p. 278.

Note that the function $x\log x-x+1$ is nonnegative and strictly convex on $[0,\infty)$. Thus, the entropy $E(\textbf{n})$ is nonnegative along the solution $\mathbf{n}(t, \cdot)$ for each $t \ge 0$. Also, the entropy dissipation $D(\textbf{n})$ is nonnegative along the solution $\mathbf{n}(t, \cdot)$ for $\alpha, \beta > 0$ such that
\begin{equation} \label{eq:ER:s2:13} \begin{aligned}
\sigma_C = \alpha k_{-}, \quad \sigma_S\sigma_E = \alpha k_{+}, \\
\sigma_C = \beta k_{p+}, \quad \sigma_E\sigma_P = \beta k_{p-}.
\end{aligned}\end{equation}
Indeed, with~(\ref{eq:ER:s2:13}) the last two integrands in~(\ref{eq:ER:s2:101}) have a form of $(x-y)(\log x - \log y)$ which is nonnegative for all $x,y \in \re_{+}$. One can choose $\alpha=1$ and $\beta=k_{-}/k_{p+}$ to obtain
\begin{equation} \label{eq:sig} \sigma_C=\sigma_E=k_{-}, \; \sigma_S=\dfrac{k_{+}}{k_{-}} \; \textnormal{ and } \; \sigma_P=\dfrac{k_{p-}}{k_{p+}},
\end{equation}
though other options are possible.

It is straightforward to verify that $D(\textbf{n}) = - \partial_t E(\textbf{n})$, which implies that $E(\textbf{n})$ is decreasing along the solution $\mathbf{n}(t, \cdot)$ and that there exists a limit of $E(\mathbf{n}(t,\cdot))$ as $t \to \infty$. By integrating this simple relation over $[t_1, \, t_2]$ ($t_2 > t_1 > 0$) we obtain
\[
E(\textbf{n}(t_1,x)) -  E(\textbf{n}(t_2,x)) = \int_{t_1}^{t_2} D(\textbf{n}(s,x)) \textnormal{d}s
\]
which implies that
\begin{equation}\label{eq:ER:s2:d5}
\lim_{t \to \infty}\int_t^{\infty} D(\textbf{n}(s,x)) \textnormal{d}s = 0.
\end{equation}
Hence, if the solution $\textbf{n}(t,x)$ tends to some $\textbf{n}_{\infty}(x)$ as $t \to \infty$, then $D(\textbf{n}_{\infty}(x))=0$ and $\textbf{n}_{\infty}$ is spatially homogeneous due to the Fisher information in~(\ref{eq:ER:s2:101}). In fact, it holds that
\begin{equation}\label{eq:ER:s2:d1} 
D(\textbf{n}(t,x)) = 0 \Longleftrightarrow  \textbf{n}(t,x) = \textbf{n}_{\infty}
\end{equation}
where $\textbf{n}_{\infty}$ is given by~(\ref{eq:ER:s2:14}) and~(\ref{eq:ER:s2:15}). Let us remark that the entropy $E(\textbf{n})$ is ``D-diffusively convex Lyapunov functional" which implies that the diffusion added to systems of ODEs is irrelevant to their long-term dynamics and that there cannot exist other (non-constant) equilibrium to~({\ref{eq:ER:s2:10}})-({\ref{eq:ER:s2:12}}) than~(\ref{eq:ER:s2:18}), \cite{Fitzgibbon-1997}. 

Further, we can write
\begin{equation}\label{eq:ER:s2:d3} 
E(\textbf{n}(t,x)) + \int_0^t D(\textbf{n}(s,x)) \textnormal{d}s = E(\textbf{n}(0,x)).
\end{equation}
Since the entropy and entropy dissipation are both nonnegative we can deduce from~(\ref{eq:ER:s2:d3}) that
\begin{equation}\label{eq:ER:s2:d4a} 
\sup_{t\in[0,\infty)} \Vert n_i \log n_i \Vert_{L^1(\Omega)} \le C,
\end{equation}
i.e., $n_i \in L^{\infty}([0,\infty); L (\log L)(\Omega))$ for each $\isecp$, and 
\begin{equation}\label{eq:ER:s2:d4c} 
\Vert \nabla \sqrt{n_i} \Vert^2_{L^{2}([0,\infty); L^2(\Omega,\re^d))} \le C,
\end{equation}
i.e., $\sqrt{n_i} \in L^{2}([0,\infty); W^{1,2}(\Omega))$ for each $\isecp$.

In addition to the above estimates, let us introduce nonnegative entropy density variables $z_i = n_i \log (\sigma_i n_i) - n_i +1/ \sigma_i$.  Then, the system~({\ref{eq:ER:s2:10}})-({\ref{eq:ER:s2:12}}) becomes
\begin{equation}\label{eq:ER:s4:01}
\dfrac{\partial z}{\partial t} - \Delta (Az) \le 0 \textnormal{\; in \;} \Omega, \quad \nabla (Az) \cdot \nu = 0 \textnormal{\; on \;} \partial \Omega
\end{equation}
where $z = \sum_i z_i$, $z_d = \sum_i D_i z_i$ (sums go through $i\in \{S,E,C,P\}$) and $A = z_d/z \in \left[ \underline{D}, \overline{D} \right]$ for $\underline{D} = \min_{i\in \{S,E,C,P\}}\{ D_i\}$ and $\overline{D} = \max_{i\in \{S,E,C,P\}}\{ D_i\}$.
Indeed, after some algebra we obtain
\begin{equation} \label{eq:ER:s4:02} \begin{aligned}
\dfrac{\partial z}{\partial t} - \Delta (Az) = &- \sum_i D_i \dfrac{\vert \nabla n_i \vert^2}{n_i} \\
& - \left( k_{+}\ns\ne - k_{-}\nc \right)\left( \log(\sigma_S \sigma_E \ns \ne) - \log(\sigma_C \nc) \right) \\ 
& - \left( k_{p-}\ne\np - k_{p+}\nc \right)\left( \log(\sigma_E \sigma_P \ne \np) - \log(\sigma_C \nc) \right)
\end{aligned} \end{equation}
where the \textit{r.h.s.} of~(\ref{eq:ER:s4:02}) is nonpositive for $\sigma_i's$ given by~(\ref{eq:sig}). The boundary condition in~(\ref{eq:ER:s4:01}) can be also easily verified.

Hence, a duality argument developed in \cite{PSch, Pierre1} and reviewed in Appendix 
 implies for each $j \in \{S,E,C,P\}$ that
\begin{equation}\label{eq:ER:s2:47}
\Vert  n_j \log (\sigma_j n_j)-n_j+1/\sigma_j \Vert_{L^2(Q_T)} \le C \left\| \sum_{i} n_i^0 \log (\sigma_i n_i^0)-n_i^0+1/\sigma_i \right\|_{L^2(\Omega)}
\end{equation}
where $C = C(\Omega, \underline{D}, \overline{D}, T)$. 
We deduce from~(\ref{eq:ER:s2:47}) that $n_i\in L^2(\log L)^2(Q_T)$ as soon as $n_i^0 \in L^2(\log L)^2(\Omega)$ for each $\isecp$. Note that a function $v \in L^2(\log L)^2(\Omega)$ is a measurable function such that $\int_{\Omega} v^2 (\log v)^2 \dx$ is finite. 

Moreover, the same duality argument implies $L^2(Q_T)$ bounds by taking into account $n_i^0 \in L^2(\Omega)$ for each $\isecp$ and $z = n_S + n_E + 2n_C + n_P$, $z_d = D_S n_S + D_E n_E + 2D_C n_C + D_P n_P$ and $A = z_d/z$ for which we directly obtain~(\ref{eq:ER:s4:01}). 

\section{Exponential convergence to equilibrium: an entropy method}
\label{sec:trend}

Let us first describe briefly a basic idea of the method. Consider an operator $A$, which can be linear or nonlinear and can involve derivatives or integrals, and an abstract problem 
\[\partial_t \rho = A \rho.\]
Assume that we can find a Lyapunov functional $E := E(\rho)$, usually called the entropy, such that $D(\rho) = -\partial_t E(\rho) \ge 0$ and
\[ D(\rho) \ge \Phi (E(\rho)-E(\rho_{eq})) \tag{EEDI} \label{eq:eedig}\]
along the solution $\rho$ where $\Phi$ is a continuous function strictly increasing from 0 and $\mathbf{\rho_{eq}}$ is a state independent of the time $t$, \cite{Arnold1, Villani-2003}. The aforementioned inequality between the entropy dissipation $D(\rho)$ and the relative entropy $E(\rho)-E(\rho_{eq})$ is known as the entropy-entropy dissipation inequality (EEDI). The EEDI and the Gronwall inequality then imply the convergence in the relative entropy $E(\rho) \to E(\rho_{eq})$ as $t \to \infty$ that can be either exponential if $\Phi(x) = \lambda x$ or polynomial if $\Phi(x) = x^{\alpha}$; in both cases $\lambda$ and $\alpha$ can be found explicitly.
In the second step, the relative entropy $E(\rho) - E(\rho_{eq})$ needs to be bounded from below by the distance $\rho-\rho_{eq}$ in some topology.   

In our reaction-diffusion setting, the relative entropy $E(\mathbf{n} \vert \mathbf{n}_{\infty}) := E(\mathbf{n})-E(\mathbf{n}_{\infty})$ for the entropy functional defined in~(\ref{eq:ER:s2:100}) can be written as
\begin{equation} \label{eq:ER:s2:24} 
 E(\mathbf{n} \vert \mathbf{n}_{\infty}) =  \sum_{i = \{S,E,C,P\}} \int_{\Omega} n_i \log \dfrac{n_i}{n_{i,\infty}} - (n_i - n_{i,\infty}) \dx \ge 0.
\end{equation}
This is a consequence of the conservation laws~(\ref{eq:ER:s2:02}) and~(\ref{eq:ER:s2:03}) which together with~(\ref{eq:ER:s2:14}) and~(\ref{eq:ER:s2:15}) imply
\begin{equation} \label{eq:ER:s2:24b} 
\sum_{i = \{S,E,C,P\}} (\overline{n_i} - n_{i,\infty}) \log (\sigma_i n_{i,\infty}) = 0.
\end{equation}
Note that the relative entropy~(\ref{eq:ER:s2:24}), known also as the Kullback-Leibler divergence, is universal in the sense that it is independent of the reaction rate constants, \cite{Gorban-2010}. The relative entropy~(\ref{eq:ER:s2:24}) can be then estimated from below by using the Czisz\'{a}r-Kullback-Pinsker (CKP) inequality known from information theory that can be stated as follows.

\begin{lemma}[Czisz\'{a}r-Kullback-Pinsker, \cite{Fellner-2016a}]  \label{theorem_ER_05}
Let $\Omega$ be a measurable domain in $\re^d$ and $u, v : \Omega \to \re_{+}$ measurable functions. Then
\begin{equation}\label{eq:ER:s2:21}
\int_{\Omega} u \log \dfrac{u}{v} - (u-v) \dx \ge \dfrac{3}{2 \Vert u \Vert_{L^1(\Omega)} + 4 \Vert v \Vert_{L^1(\Omega)}} \Vert u - v \Vert_{L^1(\Omega)}^2.
\end{equation} 
\end{lemma}
\noindent Hence, the application of the CKP inequality~(\ref{eq:ER:s2:21}) concludes the second step of the entropy method. 

Let us mention some other tools that will be later recalled in the proof of the first step. 
\begin{lemma}[Logarithmic Sobolev inequality, \cite{DF1}]
Let $\Omega  \in \re^d$ be a bounded domain such that $\vert \Omega \vert \ge 1$. Then,
\begin{equation} \label{eq:ER:s2:20}
\int_{\Omega} u^2 \log u^2 \dx - \left( \int_{\Omega} u^2 \dx \right) \log \left( \int_{\Omega} u^2 \dx \right) \le L \int_{\Omega} \vert \nabla u \vert^2
\end{equation}
that holds for some $L=L(\Omega,d)$ positive, whenever the integrals on both sides of the inequality exist.
\end{lemma}
\begin{lemma}[Poincar\'e-Wirtinger inequality, \cite{Perthame-LN}] 
Let $\Omega  \in \re^d$ be a bounded domain. Then
\begin{equation} \label{eq:PWI}
P(\Omega) \int_{\Omega} \vert u(x) - \overline{u}\vert^2 \le \int_{\Omega} \vert \nabla u \vert^2, \quad \forall u \in H^1(\Omega) 
\end{equation}
where $\overline{u} = \int_{\Omega} u(x) \dx$ and $P(\Omega)$ is the first non-zero eigenvalue of the Laplace operator with a Neumann boundary condition. 
\end{lemma}

The following lemma is a technical consequence of the Jensen inequality.

\begin{lemma} \label{lemma_ER_01} Let $\Omega \in \re^d$ be such that $\vert \Omega \vert = 1$, $u, v \in L^1(\Omega)$ be nonnegative functions, $\overline{u} = \int_{\Omega} u(x) \dx$ and $\overline{v} = \int_{\Omega} v(x) \dx$. Then 
\begin{equation}  \label{eq:ER2:01}  \begin{aligned}
& \left( \sqrt{\overline{u}} - \sqrt{\overline{v}} \right)^2 \le ( \overline{\sqrt{u}} - \sqrt{\overline{v}} )^2 + \Vert \sqrt{u} -  \overline{\sqrt{u}} \Vert^2_{L^2(\Omega)},
\end{aligned}\end{equation}
where equality occurs for $v \equiv 0$.
\end{lemma}
\begin{proof} 
Let us define an expansion of $\sqrt{u}$ around its spatial average $\overline{\sqrt{u}}$ by $\sqrt{u} = \overline{\sqrt{u}} + \delta_u(x)$ which implies immediately that $\overline{\delta_u} = 0$,
\[\Vert \sqrt{u} -  \overline{\sqrt{u}} \Vert^2_{L^2(\Omega)} = \Vert \delta_u \Vert^2_{L^2(\Omega)} = \overline{\delta_u^2} \quad \text{and} \quad \overline{u} = \overline{\sqrt{u}}^2 + \overline{\delta_u^2}.\]
Then, with the Jensen inequality $\overline{\sqrt{u}} \le \sqrt{\overline{u}}$ we can write
\[  \begin{aligned} 
\left( \sqrt{ \overline{u} } - \sqrt{\overline{v}} \right)^2 &= \overline{u} - 2 \sqrt{\overline{u}} \sqrt{\overline{v}} + \overline{v} \\
&\le \overline{\sqrt{u}}^2 - 2  \overline{\sqrt{u}} \sqrt{\overline{v}} + \overline{v} +\overline{\delta_u^2}\\
&= ( \overline{\sqrt{u}} - \sqrt{\overline{v}} )^2 + \overline{\delta_u^2}
\end{aligned}\]
which concludes the proof.
\end{proof}
\noindent In fact, with the ansatz $\sqrt{v} = \overline{\sqrt{v}} + \delta_v(x)$, we can deduce that
\[  \begin{aligned}
\left( \sqrt{\overline{u}} - \sqrt{\overline{v}} \right)^2 & \le ( \overline{\sqrt{u}} - \overline{\sqrt{v}} )^2 + \Vert \sqrt{u} -  \overline{\sqrt{u}} \Vert^2_{L^2(\Omega)} + \Vert \sqrt{v} -  \overline{\sqrt{v}} \Vert^2_{L^2(\Omega)} \\
& \le \Vert \sqrt{u} - \sqrt{v} \Vert^2_{L^2(\Omega)} + \frac{1}{P(\Omega)} (\Vert \nabla \sqrt{u} \Vert^2_{L^2(\Omega)} + \Vert \nabla \sqrt{v} \Vert^2_{L^2(\Omega)})
\end{aligned}\]
by the Jensen and Poincar\'e-Wirtinger inequalities.

Recall that we assume $\vert \Omega \vert = 1$, $\underline{D} = \min_i \{D_i\}$, $\overline{D} = \max_i \{D_i\}$ and we write shortly $\textbf{n} = (\ns,\ne,\nc,\np)$,  $\textbf{n}_{\infty} = (\nsR,\neR,\ncR,\npR)$ and $\overline{\textbf{n}}(t) = (\nsRO,\neRO,\ncRO,\npRO)$ where $\overline{n_i} = \int_{\Omega} n_i \dx$ for each $\isecp$.
In the summations we will omit $\isecp$ from the notation.

We can finally prove the exponential convergence of the solution $\textbf{n}(t)$ of~({\ref{eq:ER:s2:10}})-({\ref{eq:ER:s2:12}}) to the equilibrium $\textbf{n}_{\infty}$ given by~(\ref{eq:ER:s2:18}).

\begin{proof}{\it (of Theorem~\ref{theorem_ER_02})} We can deduce from~(\ref{eq:ER:s2:d1}) that
\[
D(\textbf{n}) = - \dfrac{\textnormal{d}}{\textnormal{d} t}E(\textbf{n}) = - \dfrac{\textnormal{d}}{\textnormal{d} t} E(\textbf{n} \vert \textbf{n}_{\infty}).
\]
As suggested above, we search for a constant $C_1$ such that
\begin{equation} \label{eq:ER:s2:26}
D(\textbf{n}) \ge  C_1 E(\textbf{n} \vert \textbf{n}_{\infty}),
\end{equation}
since in this case we obtain
\[  \dfrac{\textnormal{d}}{\textnormal{d} t}  E(\textbf{n} \vert \textbf{n}_{\infty}) \le  -C_1   E(\textbf{n} \vert \textbf{n}_{\infty}),
\]
and, by the Gronwall inequality, 
\begin{equation} \label{eq:ER:s2:26b}
E(\textbf{n} \vert \textbf{n}_{\infty})  \le  E(\textbf{n}(0,x) \vert \textbf{n}_{\infty}) e^{ -C_1 t},
\end{equation}
that is the exponential convergence in the relative entropy as $t \to \infty$.
The CKP inequality~(\ref{eq:ER:s2:21}) applied on the \textit{l.h.s.} of~(\ref{eq:ER:s2:26b}) yields
\begin{equation} \begin{aligned} \label{eq:ER:s2:25} 
E(\textbf{n} \vert \textbf{n}_{\infty}) \ge & \dfrac{1}{2M_2}\Vert \ns - \nsR \Vert^2_{L^1(\Omega)} + \dfrac{1}{M_1 + M_2}\Vert \nc - \ncR \Vert^2_{L^1(\Omega)} \\
& + \dfrac{1}{2M_1}\Vert \ne - \neR \Vert^2_{L^1(\Omega)} + \dfrac{1}{2M_2}\Vert \np - \npR \Vert^2_{L^1(\Omega)}
\end{aligned} \end{equation}
due to~(\ref{eq:ER:s2:24}) and the conservation laws~(\ref{eq:ER:s2:02}) and~(\ref{eq:ER:s2:03}).
Thus, with $C_1$ to be found and 
\[ C_2 = E(\textbf{n}(0,x) \vert \textbf{n}_{\infty}) / \min \left \{ 1/2M_1, 1/2M_2, 1/(M_1+M_2) \right \}\] 
we obtain~(\ref{eq:ER:s2:27}).

To show the EEDI~(\ref{eq:ER:s2:26}) let us split the relative entropy so that
\[
 E(\textbf{n} \vert \textbf{n}_{\infty}) = E(\textbf{n} \vert \overline{\textbf{n}}) + E( \overline{\textbf{n}} \vert \textbf{n}_{\infty}),
\]
and estimate both terms separately. For the first term we obtain that
\begin{equation} \label{eq:ER:s2:30}
 E(\textbf{n} \vert \overline{\textbf{n}}) =  \sum_{i} \int_{\Omega} n_i \log n_i \dx  - \overline{n_{i}} \log \overline{n_{i}}  \le  L \sum_{i} \int_{\Omega} \vert \nabla \sqrt{n_i} \vert^2 \dx
\end{equation}
by the logarithmic Sobolev inequality~(\ref{eq:ER:s2:20}). Hence, when compared with the entropy dissipation~(\ref{eq:ER:s2:101}), we conclude that $ D(\mathbf{n}) \ge \overline{C}_1 E(\mathbf{n} \vert \overline{\mathbf{n}})$ for the constant $\overline{C}_1 = 4 \underline{D}/L$. 

For the second term $E( \overline{\mathbf{n}} \vert \mathbf{n}_{\infty})$ we use~(\ref{eq:ER:s2:24b}) and an elementary inequality $x \log x -  x + 1 \le (x-1)^2$, which holds true for $xÊ\ge 0$, to obtain
\begin{equation}  \label{eq:ER:s2:31}  \begin{aligned}
 E( \overline{\mathbf{n}} \vert \mathbf{n}_{\infty}) & =  \sum_{i} \overline{n_i} \log \dfrac{ \overline{n_i}}{n_{i,\infty}} - \overline{n_i} + n_{i,\infty} \\
 & \le\sum_{i} \dfrac{1}{n_{i,\infty}} \left( \overline{n_i} - n_{i,\infty} \right)^2 \\
 &\le C \sum_{i} \left( \sqrt{\overline{n_i}} - \sqrt{n_{i,\infty}} \right)^2 \\
 &\le  C \left( \sum_{i} \left( \overline{\sqrt{n_i}} - \sqrt{n_{i,\infty}} \right)^2 + \sum_{i} \Vert \sqrt{n_i} -  \overline{\sqrt{n_i}} \Vert^2_{L^2(\Omega)} \right)
\end{aligned}\end{equation}
where the last inequality is due to~(\ref{eq:ER2:01}) (for $u=n_i$ and $v =\overline{v} = n_{i,\infty}$) and the constant $C = 2\max_i \{ 1/n_{i,\infty} \} \max\{ 2M_1, 2M_2, M_1+M_2\}$ is deduced from~(\ref{eq:ER:s2:02}) and~(\ref{eq:ER:s2:03}).

On the other hand, the entropy dissipation $D(\mathbf{n})$ given by~(\ref{eq:ER:s2:101}) can be estimated from below by the Poincar\'e-Wirtinger inequality~(\ref{eq:PWI})
and an elementary inequality $(x-y)(\log x - \log y) \ge 4(\sqrt{x} - \sqrt{y})^2$, which holds true for $x, y \in \re_{+}$. We obtain
\begin{equation} \label{eq:ER:s2:28} \begin{aligned}
D(\mathbf{n}) \ge \; & 4 \min \{ P(\Omega) \underline{D}, 1\} {\Big (}  \sum_{i} \Vert \sqrt{n_i} -  \overline{\sqrt{n_i}} \Vert^2_{L^2(\Omega)}  \\ & + \left. \Vert \sqrt{k_{+} \ns \ne} - \sqrt{k_{-} \nc}\Vert^2_{L^2(\Omega)} + \Vert \sqrt{k_{p-} \ne \np} - \sqrt{k_{p+} \nc}\Vert^2_{L^2(\Omega)} \right).
\end{aligned}
\end{equation}
Hence, we can conclude the proof once we find two constants $C_3$ and $C_4$ such that 
\begin{equation} \label{eq:ER:s2:eedi} \begin{aligned}
\sum_{i} & \left( \overline{\sqrt{n_i}} - \sqrt{n_{i,\infty}} \right)^2 + \sum_{i} \Vert \sqrt{n_i} -  \overline{\sqrt{n_i}} \Vert^2_{L^2(\Omega)} 
\le  C_3 \sum \Vert \sqrt{n_i} - \overline{\sqrt{n_i}} \Vert^2_{L^2(\Omega)} \\
& + C_4 \left( \Vert \sqrt{k_{+} \ns \ne} - \sqrt{k_{-} \nc}\Vert^2_{L^2(\Omega)} + \Vert \sqrt{k_{p-} \ne \np} - \sqrt{k_{p+}  \nc}\Vert^2_{L^2(\Omega)} \right),
\end{aligned} \end{equation}
since in this case, by combining~(\ref{eq:ER:s2:31})-(\ref{eq:ER:s2:eedi}), we obtain
\[
\dfrac{1}{C} E(\overline{\textbf{n}}(t) \vert \textbf{n}_{\infty}) \le \dfrac{\max \{C_3, C_4\}}{4 \min \{ P(\Omega) \underline{D}, 1\}} D(\mathbf{n}).
\]
Hence, we can derive a constant $\widetilde{C}_1$ 
such that $D(\mathbf{n}) \ge \widetilde{C}_1 E( \overline{\mathbf{n}} \vert \mathbf{n}_{\infty})$ and thus 
the convergence rate $C_1$ in the EEDI~(\ref{eq:ER:s2:26}), e.g., $C_1 = \min \{\overline{C}_1, \widetilde{C}_1\}/2$. The missing inequality~(\ref{eq:ER:s2:eedi}) is proved in Lemma~\ref{lemma_ER_06}.
\end{proof}

For the sake of simplicity, let us denote $N_i = \sqrt{n_i}$ and $N_{i,\infty} = \sqrt{n_{i,\infty}}$ and thus rewrite~(\ref{eq:ER:s2:eedi}) into the form
\begin{equation} \begin{aligned} \label{eq:ER:s2:48}
& \sum_{i } \left( \overline{N_i} - N_{i, \infty} \right)^2 + \sum_{i} \Vert N_i -  \overline{N_i} \Vert^2_{L^2(\Omega)} \le C_3 \sum_{i} \Vert N_i -  \overline{N_i} \Vert^2_{L^2(\Omega)} \\
+ C_4 & \left( \Vert \sqrt{k_{+}} N_S N_E - \sqrt{k_{-}} N_C \Vert^2_{L^2(\Omega)} + \Vert \sqrt{k_{p-}} N_E N_P - \sqrt{k_{p+}} N_C \Vert^2_{L^2(\Omega)} \right).
\end{aligned} \end{equation}

\begin{lemma} \label{lemma_ER_06} Let $N_i$, $\isecp$, be measurable functions from $\Omega$ to $\re_{+}$ satisfying the conservation laws~(\ref{eq:ER:s2:02}) and~(\ref{eq:ER:s2:03}), i.e.
\begin{equation} \label{eq:ER:s2:63} \overline{N_C^2} + \overline{N_E^2} = M_1 \quad \textnormal{and} \quad \overline{N_S^2} + \overline{N_C^2} + \overline{N_P^2} = M_2,
\end{equation}
and let $n_{i,\infty} = N_{i,\infty}^2$ be defined by~(\ref{eq:ER:s2:14}) and~(\ref{eq:ER:s2:15}). Then, there exist constants $C_3$ and $C_4$, cf.~(\ref{eq:ER1:10}) and~(\ref{eq:ER1:11}), such that~(\ref{eq:ER:s2:48}) is satisfied.
\end{lemma}

\begin{proof}
Let us use the expansion of $N_i$ around the spatial average $\overline{N_i}$ from Lemma~\ref{lemma_ER_01},   
\begin{equation}  \label{eq:ER:s2:49}
N_i = \overline{N_i} + \delta_i(x), \quad \overline{\delta}_i = 0, \quad \isecp,
\end{equation}
which implies $\overline{N_i^2} = \overline{N_i}^2 + \overline{\delta_i^2}$ for each $\isecp$ and 
\begin{equation}  \label{eq:ER:s2:51}
\sum_{i} \Vert N_i -  \overline{N_i} \Vert^2_{L^2(\Omega)} = \sum_{i} \overline{\delta_i^2}.
\end{equation}
With~(\ref{eq:ER:s2:49}) at hand, we can expand the remaining terms in~(\ref{eq:ER:s2:48}), e.g.,
\begin{equation}  \begin{aligned} \label{eq:ER:s2:52}
\Vert \sqrt{k_{+}} N_S N_E - \sqrt{k_{-}} N_C \Vert^2_{L^2(\Omega)} = & \left( \sqrt{k_{+}} \overline{N_S} \overline{N_E} - \sqrt{k_{-}} \overline{N_C}\right)^2\\
& + 2\sqrt{k_{+}}\left( \sqrt{k_{+}} \overline{N_S} \overline{N_E} - \sqrt{k_{-}} \overline{N_C}\right)  \overline{\delta_S\delta_E}\\
& + \Vert \sqrt{k_{+}}\left( \overline{N_S} \delta_E + \overline{N_E} \delta_S + \delta_S \delta_E \right) - \sqrt{k_{-}}\delta_C \Vert^2_{L^2(\Omega)}\\
\ge & \left( \sqrt{k_{+}} \overline{N_S} \overline{N_E} - \sqrt{k_{-}} \overline{N_C}\right)^2 - \sqrt{k_{+}} K_1 \sum_{i} \overline{\delta^2_i},
\end{aligned}\end{equation}
since the third term in~(\ref{eq:ER:s2:52}) is nonnegative and the second term can be estimated as follows,
\[  \begin{aligned}
2\left( \sqrt{k_{+}} \overline{N_S} \overline{N_E} - \sqrt{k_{-}} \overline{N_C}\right)  \overline{\delta_S\delta_E} &\ge -2 \left| \sqrt{k_{+}} \overline{N_S} \overline{N_E} - \sqrt{k_{-}} \overline{N_C} \right| \int_{\Omega}\delta_S\delta_E\dx \\
& \ge - K_1 ( \overline{\delta^2_S}+\overline{\delta^2_E}) \ge - K_1 \sum_{i} \overline{\delta^2_i},
\end{aligned}\]
where $K_1 = \sqrt{k_{+} M_1 M_2} + \sqrt{k_{-}(M_1 + M_2)/2}$ is deduced from the Jensen inequality $\overline{N_i^2} \ge \overline{N_i}^2$ and~(\ref{eq:ER:s2:63}). 
Analogously, we deduce for $K_2 = \sqrt{k_{p-} M_1 M_2} + \sqrt{k_{p+}(M_1 + M_2)/2}$ that
\begin{equation} \label{eq:ER:s2:55}
\Vert \sqrt{k_{p-}} N_P N_E - \sqrt{k_{p+}} N_C \Vert^2_{L^2(\Omega)} \ge  \left( \sqrt{k_{p-}} \overline{N_P} \overline{N_E} - \sqrt{k_{p+}} \overline{N_C}\right)^2 - \sqrt{k_{p-}} K_2 \sum_{i} \overline{\delta^2_i}.\end{equation}


\noindent We see that with~(\ref{eq:ER:s2:51})--(\ref{eq:ER:s2:55}) it is sufficient to find $C_3$ and $C_4$ such that 
\begin{equation} \label{eq:eedi2} \begin{aligned}
\sum_{i} ( \overline{N_i} & - N_{i, \infty})^2 + \sum_{i} \overline{\delta_i^2} \le  \left( C_3 - C_4 (\sqrt{k_{+}} K_1 + \sqrt{k_{p-}} K_2) \right) \sum_{i} \overline{\delta_i^2} \\
& + C_4 \left( \left( \sqrt{k_{+}} \overline{N_S} \overline{N_E} - \sqrt{k_{-}} \overline{N_C}\right)^2 + \left( \sqrt{k_{p-}} \overline{N_P} \overline{N_E} - \sqrt{k_{p+}} \overline{N_C}\right)^2 \right)
\end{aligned}\end{equation}
from which~(\ref{eq:ER:s2:48}) (and so~(\ref{eq:ER:s2:eedi})) directly follows.


Let us explore how far the spatial average $\overline{N_i}$ can be from the equilibrium state $N_{i,\infty}$ for each $\isecp$, i.e., let us consider a substitution
\begin{equation}\label{eq:ER1:01} \overline{N_i} = N_{i,\infty} (1+\mu_i) \end{equation}
for some $\mu_i \ge -1$, $\isecp$.  We obtain
\begin{equation}\label{eq:ER1:02}
\sum_{i} \left( \overline{N_i} - N_{i, \infty} \right)^2 = \sum_{i} \NiI^2 \mu_i^2
\end{equation}
and with~(\ref{eq:ER:s2:15}), i.e., $\sqrt{k_{+}} \NSI \NEI = \sqrt{k_{-}} \NCI$ and $\sqrt{k_{p-}} \NPI \NEI = \sqrt{k_{p+}} \NCI$,
\begin{equation}\label{eq:ER1:03} \begin{aligned}
\left( \sqrt{k_{+}} \overline{N_S} \overline{N_E} - \sqrt{k_{-}} \overline{N_C}\right)^2 = k_{-} \NCI^2 ((1+\mu_S)(1+ \mu_E) - (1+\mu_C))^2, \\
\left( \sqrt{k_{p-}} \overline{N_P} \overline{N_E} - \sqrt{k_{p+}} \overline{N_C}\right)^2 = k_{p+} \NCI^2 ((1+\mu_P)(1+ \mu_E) - (1+ \mu_C))^2.
\end{aligned} \end{equation}
Hence,~(\ref{eq:eedi2}) follows from
\begin{equation} \label{eq:eedi3} \begin{aligned}
\sum_{i} & \NiI^2 \mu_i^2  + \sum_{i} \overline{\delta_i^2} \le  \left( C_3 - C_4 (\sqrt{k_{+}} K_1 + \sqrt{k_{p-}} K_2) \right) \sum_{i} \overline{\delta_i^2} \\
& + C_4 K_3 ( \underbrace{((1+\mu_S)(1+ \mu_E) - (1+\mu_C))^2}_{{\displaystyle = I_1}} + \underbrace{((1+\mu_P)(1+ \mu_E) - (1+\mu_C))^2}_{{\displaystyle = I_2}})
\end{aligned}\end{equation}
where $K_3 = \min\{\sqrt{k_{-}}, \sqrt{k_{p+}}\}\NCI^2$.

Let us note that the conservation law~(\ref{eq:ER:s2:63}), reflecting the ansatz~(\ref{eq:ER:s2:49}) and the substitution~(\ref{eq:ER1:01}), i.e., 
\begin{eqnarray}
&\NEI^2 + \NCI^2 = \NEI^2(1+\mu_E)^2 + \delEO + \NCI^2(1+\mu_C)^2 + \delCO, \label{eq:ER1:s1} \\
& \begin{aligned} \NSI^2 + \NCI^2 + \NPI^2 = \NSI^2(1+\mu_S)^2 + \delSO & + \NCI^2(1+\mu_C)^2 + \delCO  \\ & + \NPI^2(1+\mu_P)^2 + \delPO, \label{eq:ER1:s2} \end{aligned}
\end{eqnarray}
possesses some restrictions on the signs of $\mu_i$'s. In particular, we remark that
\begin{itemize}
\item[i)] $\forall \isecp$, $-1 \le \mu_i \le \mu_{i,max}$ where $\mu_{i,max}$ depends on $\mathbf{n_{\infty}}$; 
\item[ii)] the conservation law~(\ref{eq:ER1:s1}) excludes the case when $\mu_E >0$ and $\mu_C >0$, since in this case $\NEO > \NEI$ and  $\NCO > \NCI$ and we deduce from~(\ref{eq:ER:s2:63}),~(\ref{eq:ER1:s1}) and the Jensen inequality $ \overline{N_i^2} \ge \overline{N_i}^2$, that
\[M_1 =  \overline{N_E^2} + \overline{N_C^2} \ge \NEO^2 + \NCO^2 > \NEI^2 + \NCI^2 = M_1,\]
which is a contradiction;
\item[iii)] analogously, the conservation law~(\ref{eq:ER1:s2}) excludes the case when $\mu_S >0$, $\mu_C >0$ and $\mu_P >0$;
\item[iv)] for $-1 \le \mu_E, \mu_C \le 0$, the conservation law~(\ref{eq:ER1:s1}) implies $\NEI^2 \mu_E^2 + \NCI^2 \mu_C^2 \le \sum \deliO$, since for $-1 \le s \le 0$ we have $-1 \le s \le -s^2 \le 0$ and we can deduce from~(\ref{eq:ER1:s1}) that
\[ \begin{aligned} 0 = \NEI^2(2\mu_E + \mu_E^2) + \NCI^2(2\mu_C & + \mu_C^2) + \delCO + \delEO \\
&\le - \NEI^2 \mu_E^2 - \NCI^2 \mu_C^2 + \sum \deliO;
\end{aligned}\]
\item[v)] analogously, for $-1 \le \mu_S, \mu_C, \mu_P \le 0$, the conservation law~(\ref{eq:ER1:s2}) implies that $\NSI^2 \mu_S^2 + \NCI^2 \mu_C^2 + \NPI^2 \mu_P^2 \le \sum \deliO$.
\end{itemize}

To find $C_3$ and $C_4$ explicitly, we have to consider all possible configurations of $\mu_i$'s in~(\ref{eq:eedi3}), that is all possible quadruples $(\mu_E, \mu_C, \mu_S, \mu_P)$ depending on their signs. The remarks~(ii) and~(iii) reduce the total number of quadruples by five and the remaining 11 quadruples are shown in Table~\ref{tab:ER1:1}.\\

\renewcommand{\arraystretch}{1.4}
\begin{table}[h]
\centering
\caption[Relations between $\mu_i$ for $\isecp$]{\label{tab:ER1:1} Eleven quadruples of possible relations among $\mu_i$, $\isecp$, which are allowed by the conservation laws~(\ref{eq:ER1:s1}) and~(\ref{eq:ER1:s2}). In the table $``+"$ means that $\mu_i > 0$ and $``-"$ that $-1 \le \mu_i \le 0$. Each quadruple is denoted by a Roman numeral from I to XI.}
\label{tab:ER1:1}
\begin{tabular}{c|c|c|c|c|c|c|c|c|c|c|c}
$\mu_E$ & \multicolumn{4}{c|}{$-$} & \multicolumn{4}{c|}{$+$} & \multicolumn{3}{c}{$-$} \\ \hline
$\mu_C$ & \multicolumn{4}{c|}{$-$} & \multicolumn{4}{c|}{$-$} & \multicolumn{3}{c}{$+$}\\ \hline
$\mu_S$ & $-$ & $-$     & $+$      & $+$    & $-$     & $-$     & $+$  & $+$ & $-$     & $-$      & $+$  \\ \hline
$\mu_P$ & $-$ & $+$     & $-$      & $+$    & $-$     & $+$     & $-$   & $+$ & $-$   & $+$     & $-$   \\ \hline
  & (I) & (II)   & (III)  & (IV)  & (V)  & (VI)  & (VII) & (VIII) & (IX)  & (X) & (XI) 
\end{tabular}
\end{table}
\renewcommand{\arraystretch}{1}

Ad (I). The remarks (iv) and (v) implies $\sum_i \NiI^2 \mu_i^2 \le 2 \sum_i \deliO$ and, therefore,~(\ref{eq:eedi3}) is satisfied for $C_3=3$  and $C_4=0$.\\

Ad (II) and (III). We prove~(\ref{eq:eedi3}) for $-1 \le \mu_E, \mu_C \le 0$ and $\mu_S$ and $\mu_P$ having opposite signs. Firstly, let us remark that~(\ref{eq:ER1:s1}) implies that
\[ \NEI^2 = \NEI^2(1+\mu_E)^2 + \NCI^2(2\mu_C + \mu_C^2) + \delEO + \delCO,\]
i.e.,
\[\begin{aligned} (1+\mu_E)^2 & = 1 - \frac{\NCI^2}{\NEI^2}(2\mu_C + \mu_C^2) - \frac{1}{\NEI^2} (\delEO + \delCO) \\
& \ge 1 - \frac{1}{\NEI^2}  \sum \deliO \end{aligned},\]
since for $-1 \le \mu_C \le 0$ there is $-1 \le 2\mu_C + \mu_C^2 \le 0$. Then, by using an elementary inequality $a^2+b^2 \ge (a-b)^2/2$ we obtain for
\[I_1 + I_2 = ((1+\mu_S)(1+ \mu_E) - (1+\mu_C))^2 + ((1+\mu_P)(1+ \mu_E) - (1+ \mu_C))^2\]
that
\begin{equation} \label{eq:ER1:12} \begin{aligned} I_1 + I_2 & \ge \frac{1}{2}(\mu_S-\mu_P)^2(1+\mu_E)^2 \\
& \ge  \frac{1}{2} K_4 (\NSI^2 \mu_S^2 + \NPI^2 \mu_P^2) - K_5 \sum \deliO
\end{aligned} \end{equation}
since $\mu_S$ and $\mu_P$ have opposite signs and are bounded above by $\mu_{S,max}$ and $\mu_{P,max}$ (by the remark~(i)). In~(\ref{eq:ER1:12}), $K_4 = \min \left\{1/\NSI^2, 1/\NPI^2 \right\}$ is sufficient, nevertheless, we will take 
\begin{equation}\label{eq:ER1:09} K_4 = \min_{\isecp} \left\{\dfrac{1}{\NiI^2} \right\} \textnormal{ and } K_5 =  \dfrac{1}{\NEI^2} (\mu_{S,max}^2 + \mu_{P,max}^2),\end{equation}
since $K_4$ in~(\ref{eq:ER1:09}) will appear several times elsewhere. Together with $\NEI^2 \mu_E^2 + \NCI^2 \mu_C^2 \le \sum \deliO$ for $-1 \le \mu_E, \mu_C \le 0$ known from the remark~(iv), we deduce
\begin{equation}\label{eq:ER:t1} \sum \NiI^2 \mu_i^2 + \sum \deliO \le 2 \left(1 + \frac{K_5}{K_4}\right) \sum \deliO + \frac{2}{K_4} (I_1 + I_2), \end{equation}
and we see that~(\ref{eq:eedi3}) is satisfied for
\[C_4 = \frac{2}{K_3 K_4} \quad \textnormal{and} \quad C_3 = 2\left(1 + \frac{K_5}{K_4}\right) + C_4 \left( \sqrt{k_{+}} K_1 + \sqrt{k_{p-}} K_2 \right), \]
when~(\ref{eq:eedi3}) is compared with the \textit{r.h.s.} of~(\ref{eq:ER:t1}).\\
 
Ad (IV) Assume $-1 \le \mu_E, \mu_C \le 0$ and $\mu_S, \mu_P > 0$.  A combination of~(\ref{eq:ER1:s1}) and~(\ref{eq:ER1:s2}) gives
\begin{equation}\label{eq:ER3:04} \begin{aligned} \NEI^2 - \NSI^2 - \NPI^2 &=  \NEO^2 + \delEO - \NSO^2 - \delSO - \NPO^2 - \delPO \\
& \le \NEI^2 - \NSO^2 - \NPO^2 + \delEO - \delSO - \delPO, 
\end{aligned}\end{equation}
since $\NEO^2 \le \NEI^2$ for $-1 \le \mu_E \le 0$. We deduce from~(\ref{eq:ER3:04}) that
\[- \NSI^2 - \NPI^2 \le  - \NSI^2(1+\mu_S)^2 - \NPI^2(1+\mu_P)^2 + \delEO - \delSO - \delPO \]
and
\[\NSI^2(2\mu_S + \mu_S^2) + \NPI^2(2\mu_P + \mu_P^2) \le \delEO - \delSO - \delPO  \le \sum \deliO. \]
Thus, $\NSI^2 \mu_S^2 + \NPI^2 \mu_P^2 \le \sum \deliO$ since $\mu_S, \mu_P > 0$.
This estimate together with the remark~(iv) yields $\sum \NiI^2 \mu_i^2 \le  2\sum \deliO$. Similarly as in the case (I),~(\ref{eq:eedi3}) is satisfied for $C_3=3$  and $C_4=0$.
\\

Ad (V) Let us now consider the case when $-1 \le \mu_S, \mu_C, \mu_P \le 0$ and $\mu_E>0$. As in the case (IV), a combination of~(\ref{eq:ER1:s1}) and~(\ref{eq:ER1:s2}) gives
\begin{equation}\label{eq:ER3:03} \begin{aligned} \NSI^2 + \NPI^2  - \NEI^2 &= \NSO^2 + \delSO + \NPO^2 + \delPO -  \NEO^2 - \delEO \\
& \le \NSI^2 + \NPI^2 -  \NEO^2 + \delSO + \delPO - \delEO, \end{aligned}\end{equation}
since, again, $\NiO^2 \le \NiI^2$ for $-1 \le \mu_i \le 0$. Hence, for $\mu_E>0$ we deduce from~(\ref{eq:ER3:03}) that $ \NEI^2 \mu_E^2 <  \sum \deliO$, which with the remark~(v) gives $\sum \NiI^2 \mu_i^2 <  2\sum \deliO$. Thus,~(\ref{eq:eedi3}) is satisfied for $C_3=3$  and $C_4=0$.\\

Ad (VI) and (VII). Assume that $\mu_E > 0$, $-1 \le \mu_C \le 0$ and $\mu_S$ and $\mu_P$ have opposite signs. Then using an elementary inequality $a^2+b^2 \ge (a+b)^2/2$ we obtain
\[I_1+I_2 \ge \dfrac{1}{2}(\mu_S-\mu_P)^2(1+\mu_E)^2 > \dfrac{1}{2}(\mu_S-\mu_P)^2 \ge  \dfrac{1}{2}(\mu_S^2 + \mu_P^2),\]
since $(1+\mu_E)^2>1$ and $\mu_S$ and $\mu_P$ have opposite signs. Further, it holds that $(1+\mu_k)(1+\mu_E) > (1+\mu_E)$ for $\mu_k$ being either $\mu_S>0$ or $\mu_P>0$ (one of them is positive). This implies $(1+\mu_k)(1+\mu_E) - (1+\mu_C) > \mu_E - \mu_C > 0$
and thus ($\mu_E$ and $\mu_C$ have opposite signs)
\[I_1 + I_2 >  (\mu_E-\mu_C)^2 \ge \mu_E^2 + \mu_C^2.\]
Altogether, we obtain for both cases that $I_1 + I_2 > \sum \mu_i^2/4 \ge K_4/4 \sum \NiI^2 \mu_i^2$ where $K_4$ is defined in~(\ref{eq:ER1:09}). 
We deduce that~(\ref{eq:eedi3}) is satisfied for 
\begin{equation} \label{eq:es1} C_4 = \frac{4}{K_3 K_4} \quad \textnormal{and} \quad C_3 = 1 + C_4 \left( \sqrt{k_{+}} K_1 + \sqrt{k_{p-}} K_2\right). \end{equation}

Ad (VIII). Assume that $\mu_E, \mu_S, \mu_P > 0$ and $-1 \le \mu_C \le 0$. Using the similar arguments as in the previous case, in particular, $(1+\mu_S)(1+\mu_E) > (1+\mu_S)$, $(1+\mu_S)(1+\mu_E) > (1+\mu_E)$, $(1+\mu_P)(1+\mu_E) > (1+\mu_P)$ and $(1+\mu_P)(1+\mu_E) > (1+\mu_E)$ and since $\mu_i-\mu_C > 0$ for each $i \in \{S,E,P\}$,  we can write
\[\begin{aligned}I_1 + I_2 &= ((1+\mu_S)(1+ \mu_E)-(1+\mu_C))^2 + ((1+\mu_P)(1+ \mu_E)-(1+\mu_C))^2 \\
&\ge \dfrac{1}{2}(\mu_S-\mu_C)^2 + (\mu_E-\mu_C)^2 + \dfrac{1}{2}(\mu_P-\mu_C)^2 \ge \dfrac{1}{2} \sum \mu_i^2 \ge \dfrac{K_4}{2} \sum \NiI^2 \mu_i^2.
 \end{aligned}\]
Hence,~(\ref{eq:eedi3}) is satisfied for $C_4 = 2/K_3K_4$ and $C_3$ defined in~(\ref{eq:es1}).
\\

Ad (IX). The case when $-1 \le \mu_E, \mu_S, \mu_P \le 0$ and $\mu_C > 0$ is similar to the case (VIII). Now we observe that $\mu_C-\mu_i > 0$ for each $i \in \{S,E,P\}$ and that $(1+\mu_S)(1+\mu_E) \le (1+\mu_S)$, $(1+\mu_S)(1+\mu_E) \le (1+\mu_E)$, $(1+\mu_P)(1+\mu_E) \le (1+\mu_P)$ and $(1+\mu_P)(1+\mu_E) \le (1+\mu_E)$ which can be used to conclude $I_1 + I_2 \ge \sum \mu_i^2/2\ge K_4/2 \sum \NiI^2 \mu_i^2$. 
The constants $C_3$ and $C_4$ are the same as in the case (VIII).
\\

Ad (X). Assume that $-1 \le \mu_E \le 0$, $\mu_C > 0$, $-1 \le \mu_S \le 0$ and $\mu_P > 0$. By the same argument as in~(IX), we can write
\begin{equation} \label{eq:ER1:05} I_1 + I_2 \ge I_1 \ge (\mu_C-\mu_E)^2 \ge \mu_C^2 + \mu_E^2.\end{equation}
Using the same elementary inequality as in (II) and (VI), we obtain
\begin{equation} \label{eq:ER1:06} I_1 + I_2 \ge \dfrac{1}{2}(\mu_S-\mu_P)^2(1+\mu_E)^2, \end{equation}
where $-1 \le \mu_E \le 0$, thus we cannot proceed in the way as in the cases~(VI) and (VII) nor in the cases~(II) and (III), since $\mu_C$ is positive now. Nevertheless, we distinguish two subcases when $-1 < \eta \le \mu_E \le 0$ and $-1 \le \mu_E < \eta$. For example, $\eta = -1/2$ works well, however, a more suitable constant $\eta$ could be possibly found. For $\eta = -1/2$ and $\eta \le \mu_E \le 0$ we obtain from~(\ref{eq:ER1:06}) that 
\begin{equation} \label{eq:ER1:07} I_1 + I_2 \ge \dfrac{1}{8}(\mu_S-\mu_P)^2 \ge \dfrac{1}{8}(\mu_S^2+\mu_P^2). \end{equation}
This with~(\ref{eq:ER1:05}) implies that $I_1 + I_2 \ge \sum \mu_i^2/16 \ge K_4/16 \sum \NiI^2 \mu_i^2$ and
we conclude that~(\ref{eq:eedi3}) is satisfied for $C_4 = 16/K_3K_4$ and $C_3$ defined in~(\ref{eq:es1}).

For $\eta = -1/2$ and $-1 \le \mu_E < \eta$ we obtain, by using an elementary inequality $(a-b)^2 \ge a^2/2-b^2$, that
\begin{equation} \label{eq:ER1:08} \begin{aligned} I_1 + I_2 \ge I_1 & = ((1+\mu_C) - (1+\mu_S)(1+\mu_E))^2 \\
& \ge \dfrac{1}{2} (1+\mu_C)^2 - (1+\mu_S)^2(1+\mu_E)^2 > \dfrac{1}{4}, \end{aligned} \end{equation}
since $(1+\mu_C)^2 > 1$ for $\mu_C > 0$ and $(1+\mu_S)^2(1+\mu_E)^2 < 1/4$ for $-1 \le \mu_S \le 0$ and $-1 \le \mu_E < -1/2$. On the other hand, $\sum \NiI^2 \mu_i^2 \le \sum \NiI^2 \mu_{i,\max}^2$ by the remark~(i). In fact, for the given quadruple of $\mu_i$'s, we deduce from~(\ref{eq:ER1:s2}) a constant $K_6 = \NSI^2 (1 + \NPI^2 + \NCI^2) + \NEI^2$ such that $\sum \NiI^2 \mu_i^2 \le K_6$. We see that~(\ref{eq:eedi3}) is satisfied for $C_4 = K_6/4K_3$ and $C_3$ as in~(\ref{eq:es1}).
\\

Ad (XI). Finally, assume that $-1 \le \mu_E \le 0$, $\mu_C > 0$, $\mu_S > 0$ and $-1 \le \mu_P \le 0$. This case is symmetric to the previous case (X), thus the same procedure can be applied again (it is sufficient to exchange superscripts $S$ and $P$ everywhere they appear in (X)) to deduce the constants $C_3$ and $C_4$ in~(\ref{eq:eedi3}). In particular, for $ -1/2 \le \mu_E \le 0$ we take $C_4 = 16/K_3K_4$ and for $ -1 \le \mu_E < -1/2 $ we take $C_4 = K_7/4K_3$ and $K_7 = \NPI^2 (1 + \NSI^2 + \NCI^2) + \NEI^2$. In both subcases $C_3$ is as in~(\ref{eq:es1}).\\

From the eleven cases (I)-(XI), we need to take
\begin{equation} \label{eq:ER1:10} C_4 = \frac{1}{K_3} \max \left\{\frac{16}{K_4}, \frac{K_6}{4}, \frac{K_7}{4} \right\}\end{equation}
and
\begin{equation} \label{eq:ER1:11}C_3 = \max \left\{3, 2\left(1 + \frac{K_5}{K_4}\right)  \right\} + C_4 \left( \sqrt{k_{+}} K_1 + \sqrt{k_{p-}} K_2 \right) \end{equation}
to find~(\ref{eq:eedi3}) true and thus to conclude the proof.
\end{proof}




\section*{Appendix. A duality principle}
\addcontentsline{toc}{section}{Appendix. A duality principle}

We recall a duality principle \cite{Pierre1, PSch} which is used to show $L^2(\log L)^2$ and $L^2$ bounds, respectively, for the solution to~(\ref{eq:ER:s2:10})-(\ref{eq:ER:s2:12}). Note that a more general result is proved in \cite{Pierre1}, Chap. 6, than presented here.

\begin{lemma}[Duality principle] 
\label{lemma_ER_03} 
Let $0<T<\infty$ and $\Omega$ be a bounded, open and regular (e.g., $C^2$) subset of $\re^d$. Consider a nonnegative weak 
solution $u$ of the problem
\begin{equation}\label{eq:ER:s2:40}
\left\{ \begin{aligned} & \partial_t u - \Delta (A u) \le 0, \\
& \nabla (A u) \cdot \nu = 0, \quad \forall t \in I, \; x \in \partial \Omega,  \\
& u(0,x) = u_0(x), \\
\end{aligned} \right.
\end{equation}
where we assume that $0 < A_1 \le A = A(t,x) \le A_2 < \infty $ is smooth, $A_1$ and $A_2$ are strictly positive constants, $u_0 \in L^2(\Omega)$ and $\int u_0 \ge 0$. Then,
\begin{equation}\label{eq:ER:s2:41}
\Vert u \Vert_{L^2(Q_T)} \le C \Vert u_0 \Vert_{L^2(\Omega)}
\end{equation}
where $C=C(\Omega, A_1, A_2, T)$.
\end{lemma}
\begin{proof} Let us consider an adjoint problem: find a nonnegative function $v \in C(I;L^2(\Omega))$ which is regular in the sense that $\partial_t v, \Delta v \in L^2(Q_T)$ and satisfies
\begin{equation}\label{eq:ER:s2:42}
\left\{ \begin{aligned} & - \partial_t v - A \Delta v = F, \\
& \nabla v \cdot \nu = 0, \quad \forall t \in I, \; x \in \partial \Omega,  \\
& v(T,x) = 0, \\
\end{aligned} \right.
\end{equation}
for $F = F(t,x) \in L^2(Q_T)$ nonnegative. The existence of such $v$ follows from the classical results on parabolic equations \cite{Lady1968}. 

By combining equations for $u$ and $v$, we can readily check that
\[ -\dfrac{\textnormal{d}}{\dt} \int_{\Omega} u v \ge \int_{\Omega} u F \]
which, by using $v(T)=0$, yields
\begin{equation}\label{eq:ER:s2:43}
\int_{Q_T} u F \le \int_{\Omega}  u_0 v_0. 
\end{equation}
By multiplying equation for $v$ in~(\ref{eq:ER:s2:42}) by $-\Delta v$, integrating per partes and using the Young inequality, we obtain
\[ -\dfrac{1}{2}\dfrac{\textnormal{d}}{\dt} \int_{\Omega} \vert \nabla v \vert^2 + \int_{\Omega} A (\Delta v)^2  = - \int_{\Omega} F \Delta v \le \int_{\Omega} \dfrac{F^2}{2A} + \dfrac{A}{2}(\Delta v)^2, \]
i.e.
\[ -\dfrac{\textnormal{d}}{\dt} \int_{\Omega} \vert \nabla v \vert^2 + \int_{\Omega} A (\Delta v)^2 \le \int_{\Omega} \dfrac{F^2}{A}. \]
Integrating this over $[0,T]$ and using $v(T)=0$ gives
\[ \int_{\Omega} \vert \nabla v_0 \vert^2 + \int_{Q_T} A (\Delta v)^2 \le \int_{Q_T} \dfrac{F^2}{A}. \]
Thus we obtain the a-priori bounds
\begin{equation}\label{eq:ER:s2:44}
\Vert \nabla v_0 \Vert_{L^2(\Omega,\re^d)} \le \left\| \dfrac{F}{\sqrt{A}} \right\|_{L^2(Q_T)} \quad \textnormal{and} \quad \Vert \sqrt{A} \Delta v \Vert_{L^2(\Omega)} \le \left\| \dfrac{F}{\sqrt{A}} \right\|_{L^2(Q_T)}.
\end{equation}
From the equation for $v$ we can write (again, by integrating this equation over $\Omega$ and $[0,T]$ and using $v(T)=0$)
\[ \int_{\Omega} v_0 =  \int_{Q_T} A \Delta v + F. \]
Hence,
\begin{equation}\label{eq:ER:s2:45} \begin{aligned}
 \int_{\Omega} v_0 =  \int_{Q_T} \sqrt{A} \left( \sqrt{A} \Delta v + \dfrac{F}{\sqrt{A}}\right) & \le \Vert \sqrt{A} \Vert_{L^2(Q_T)} \left\| \sqrt{A}\Delta v +  \dfrac{F}{\sqrt{A}} \right\|_{L^2(Q_T)} \\
 & \le 2 \Vert \sqrt{A} \Vert_{L^2(Q_T)} \left\| \dfrac{F}{\sqrt{A}} \right\|_{L^2(Q_T)},
\end{aligned} \end{equation}
which follows from the H\"{o}lder inequality and~(\ref{eq:ER:s2:44}).

To conclude the proof, let us return to~(\ref{eq:ER:s2:43}) and write
\[ \begin{aligned}
0 \le \int_{Q_T} u F \le \int_{\Omega}  u_0 v_0 &= \int_{\Omega}  u_0( v_0 - \overline{v_0}) + u_0\overline{v_0}\\
& \le \Vert u_0 \Vert_{L^2(\Omega)}\Vert  v_0 - \overline{v_0} \Vert_{L^2(\Omega)} + \int_{\Omega} \overline{u_0} v_0 \\
& \le C(\Omega) \Vert u_0 \Vert_{L^2(\Omega)} \Vert \nabla v_0  \Vert_{L^2(\Omega,\re^d)} + \overline{u_0} \int_{\Omega} v_0,
\end{aligned}\]
where we have used the H\"{o}lder and Poincar\'{e}-Wirtinger
inequalities, respectively. Recall that $\overline{v} = \dfrac{1}{\vert \Omega \vert} {\displaystyle \int_{\Omega} v \dx}.$ The norm of the gradient $v_0$ can be estimated by~(\ref{eq:ER:s2:44}) and the last remaining integral by~(\ref{eq:ER:s2:45}) so that we obtain
\begin{equation}\label{eq:ER:s2:46} 
\int_{Q_T} u F \le \left( C(\Omega) \Vert u_0 \Vert_{L^2(\Omega)} + 2\overline{u_0} \Vert \sqrt{A} \Vert_{L^2(Q_T)} \right) \left\| \dfrac{F}{\sqrt{A}} \right\|_{L^2(Q_T)},
\end{equation}
which holds true for any $F \in L^2(Q_T)$. Thus, for $F = Au$ we can finally write
\begin{equation}\label{eq:ER:s2:46} 
\Vert \sqrt{A}u\Vert_{L^2(Q_T)} \le C(\Omega) \Vert u_0 \Vert_{L^2(\Omega)} + 2 \overline{u_0} \Vert \sqrt{A} \Vert_{L^2(Q_T)}
\end{equation}
and deduce~(\ref{eq:ER:s2:41}) by using the boundedness of $A$, i.e. $A_1 \le A(t,x) \le A_2$. 
\end{proof}

\section*{Acknowledgements}
This work was partially supported by a public grant as part of the Investissement d'avenir project, reference ANR-11-LABX-0056-LMH, LabEx LMH. The author would like to thank to Bao Tang and Beno\^{i}t Perthame for useful discussions and suggestions.

\bibliography{biblio}

\end{document}